\documentclass[default]{svmult}
\usepackage{amscd,amsmath,pxfonts,amssymb,amstext}
\usepackage{pxfonts}
\newtheorem{theo} {Theorem} [section]
\newtheorem{lem} [theo]{Lemma}
\newtheorem{prop}[theo]{Proposition}
\newtheorem{cor} [theo]{Corollary}

\title*{Cycle Connectivity and Automorphism Groups of Flag Domains}
\author{Alan Huckleberry \\ \vspace{0.5cm} \textit{Dedicated to Arcady Onishchik on the occasion of his 80${}^{th}$ birthday.}}
\institute{Alan Huckleberry \at Institut f\"ur Mathematik Ruhr Universit\"at \\ Bochum Universit\"atsstrasse 150 \\ 44780 Bochum, Germany \\ and \\ Jacobs University Bremen, School of Engineering and Science\\ Campus Ring 1, 28759, Bremen, Germany\\ \email{ahuck@gmx.de}} 

\begin{document}

\maketitle

\vspace{1.5cm}

\abstract{\\A flag domain $D$ is an open orbit of a real form $G_0$ in a flag manifold 
$Z=G/P$ of its complexification. If $D$ is holomorphically convex, then, since
it is a product of a Hermitian symmetric space
of bounded type and a compact flag manifold, $\mathrm{Aut}(D)$ is easily described.
If $D$ is not holomorphically convex, then in previous work it was shown that
$\mathrm{Aut}(D)$ is a Lie group whose connected component at the identity 
agrees with $G_0$ except possibly in situations 
which arise in Onishchik's list of flag manifolds where $\mathrm{Aut}(Z)^0=\widehat{G}$ 
is larger than $G$.  In the present work the group $\mathrm{Aut}(D)^0=\widehat{G}_0$ 
is described as a real form of $\widehat{G}$.  Using an observation of Kollar,
new and much simpler proofs of much of our previous work in the case where 
$D$ is not holomorphically convex are given. \vspace{0.3cm}  \vspace{0.3cm} }

\section{Introduction and statement of results}\label{intro}
\noindent
Recall that if $Z$ is a compact complex manifold, then its Lie algebra
$\mathfrak {g}=\mathrm{Vect}_{\mathcal {O}}(Z)$ of holomorphic vector
fields is finite-dimensional and that the fields in $\mathfrak {g}$ can be 
integrated to define a holomorphic action of the associated simply-connected complex
Lie group $G$.  If $Z$ is homogeneous in the sense that this group acts
transitively,  then we choose a base point $z_0\in Z$, let $H=G_{z_0}$ denote
the isotropy group at that point and identify $Z$ with the quotient $G/H$.  If
$G$ is projective algebraic with trivial Albanese, i.e., with $b_1(Z)=0$, then
$G$ is semisimple, the isotropy group $H$ is a so-called parabolic subgroup,
which from now on we denote by $P$, and $Z=G/P$ is a $G$-orbit in the
projective space $\mathbb {P}(V)$ of an appropriate  $G$-representation space $V$.  In
this case we refer to $Z$ as a flag manifold.

\bigskip\noindent
A real form $G_0$ of $G$ is a real Lie subgroup of $G$ such
that the complexification $\mathfrak{g}_0+i\mathfrak{g}_0$ is the Lie algebra
$\mathfrak{g}$.  If $Z=G/P$ is a flag manifold, then any real form $G_0$ of $G$
has only finitely many orbits in $Z$ (\cite{W}, see also \cite{FHW} for this as well
as other background.).  In particular, $G_0$ always has at least one
open orbit $D$.  We refer to such an open orbit as a flag domain.  If $G_0$ is not simple, then,
$D$ has product structure corresponding to the factors of $G_0$.  Thus, for our
considerations here there is no loss of generality in assuming that $G_0$ is simple
which we do throughout.  Note that if $G_0$ has the abstract structure of a complex
Lie group, then its complexification $G$ is, however, not simple.  Note also that
$G_0$ could act transitively on $Z$, e.g., this is always the case for a compact
real form.  However, from the point of view of this article, in that case all phenomena
are well understood and therefore we assume that $D$ is a proper subset of $Z$.

\bigskip\noindent
Since by assumption a flag domain $D$ is noncompact, there is no a priori reason
to expect that $\mathrm{Aut}(D)$ or $\mathrm{Vect}_{\mathcal {O}}(D)$ is finite-dimensional.
In fact if $D$ possesses non-constant holomorphic functions, the latter  is not the case
and the former is often not the case as well.  Let us begin here by reviewing this situation.
 
\bigskip\noindent
If $X$ is any complex manifold, then the equivalence relation,
$$
x\sim y \ \Leftrightarrow \ f(x)=f(y) \ \text{for all} \ f\in \mathcal{O}(X)\,,
$$
is equivariant with respect to the full group $\mathrm{Aut}(X)$ of holomorphic automorphisms.
If $X=G/H$ is homogeneous with respect to a Lie group
of holomorphic transformations, then the reduction $X\to X/\sim $ by this
equivalence relation is a $G$-equivariant holomorphic homogeneous fibration
$G/H\to G/I$.  If $D=G_0/H_0$ is a flag domain, then this reduction has a particularly 
simple form (\cite{W}, \cite{FHW},$\S4.4$ ).  For this let $D=G_0.z_0$ with
$H_0$ (resp $P$) be the $G_0$-isotropy subgroup (resp. $G$-isotropy subgroup)
at $z_0$.
\begin {theo}
If $D=G_0.z_0$ is a flag domain with $\mathcal{O}(D)\not=\mathbb {C}$, then
the holomorphic reduction $D=G_0/H_0\to G_0/I_0=\widetilde {D}$ is the restriction
of a fibration $Z=G/P\to G/\widetilde{P}=\widetilde{Z}$ of the ambient flag manifold with the
properties
\begin {enumerate}
\item
The fiber of $Z\to \widetilde{Z}$, which is itself is a flag manifold, agrees with
the fiber of $D\to \widetilde {D}$.
\item
The base $\widetilde{D}$ is a $G_0$-flag domain in $\widetilde{Z}$. It is
a Hermitian symmetric space of noncompact type embedded in a
canonical way in its compact dual $\widetilde {Z}$.
\end {enumerate}
\end {theo}
\noindent
Recall that a symmetric space of noncompact type of a simple Lie group is a topological cell and
that in the Hermitian case it is a Stein manifold.  Thus  Grauert's Oka-principle
implies that the fibration $D\to \widetilde {D}$ is a (holomorphically) trivial bundle.  
As a consequence we have the following more refined version of the above result.
\begin {cor}
A flag domain $D$ with $\mathcal{O}(D)\not=\mathbb {C}$ is the product 
$\widetilde{D}\times F$  of a Hermitian symmetric space $\widetilde{D}$ 
of noncompact type and a compact flag manifold $F$.  In particular, $D$
is holomorphically convex and $D\to \widetilde {D}$ is its Remmert reduction.
\end  {cor}
\noindent
As indicated above our goal here is to describe the connected component at
the identity $\mathrm{Aut}(D)^0$ of the group of holomorphic automorphisms of
any given flag domain $D$.  With certain exceptions which we cover
in detail below, we carried out this project in \cite{H1} by studying the
associated action of $\mathrm{Aut}(D)$ on a certain space (described below)
$\mathcal{C}_q(D)$ of holomorphic cycles. If $D=\widetilde{D}$ is a 
Hermitian symmetric space of noncompact type,  such cycles are
just isolated points and $\mathcal{C}_q(D)=D$.  Thus the cycle space gives
us no additional information.  However, in this case $D$ possesses the
invariant Bergman metric and as a result $\mathrm{Aut}(D)$ is well-understood.

\bigskip\noindent
If $D=\widetilde{D}\times F$ is a product with nontrivial base and fiber, then,
although it is infinite-dimensional, $\mathrm{Aut}(D)$ is in a certain sense easy
to describe: The fibration $D\to \widetilde{D}$ induces a surjective homomorphism
$\mathrm{Aut}(D)\to \mathrm{Aut}(\tilde {D})$.  The kernel is the space 
$\mathrm{Hol}(\widetilde{D},\mathrm{Aut}(F))$ of holomorphic maps from
the base to the complex Lie group $\mathrm{Aut}(F)$ and as a result 
$\mathrm{Aut}(D)=\mathrm{Hol}(\widetilde{D},\mathrm{Aut}(F))\rtimes \mathrm{Aut}(D)$
has semidirect product structure. 

\bigskip\noindent
Having settled the case where $\mathcal{O}(D)\not=\mathbb {C}$, or equivalently
where $D$ is holomorphically convex, we turn to the situation where
$\mathcal{O}(D)=\mathbb {C}$.   In \cite{H1} we showed that $\mathrm{Aut}(D)$
is a (finite-dimensional) Lie group which, with certain exceptions which are handled
below, $\mathrm{Aut}(D)^0=G_0$.  Other than taking care of these exceptional
cases, where in fact $\mathrm{Aut}(D)^0$ contains $G_0$ as a proper Lie subgroup,
here we also make use of an observation of Kollar (\cite{K}) which
leads to a simple proof of $\mathrm{Aut}(D)^0=G_0$ with the possible exceptions.
This proof is given in $\S\ref{connectedness}$.

\bigskip\noindent
Before going into the details of proofs, let us state the main result of the paper.
For this the following classification theorem of A. Onishchik (\cite{O1}) is the
key first step for handling the exceptional cases mentioned above.
\begin {theo}
The following is a list of the flag manifolds $Z$ and (connected) complex simple
Lie groups $G$ and $\widehat{G}$ so that $Z=G/P$ and $\widehat{G}:=\mathrm{Aut}(Z)^0$
properly contains $G$.
\begin {enumerate}
\item
The manifold $Z$ is the odd-dimesional projective space $\mathbb {P}(\mathbb {C}^{2n})$
where, after lifting to simply-connected coverings, 
$G=\mathrm{Sp}_{2n}(\mathbb {C})$ and $\widehat{G}=\mathrm{SL}_{2n}(\mathbb {C})$
\item
The $5$-dimensional complex quadric $Z$ is equipped with the standard action of
$\widehat{G}=\mathrm{SO}_7(\mathbb {C})$ and $G$ is the exceptional complex Lie group
$G_2$ embedded in $\widehat{G}$ as the automorphism group of the octonians.
\item
Equipping $\mathbb {C}^{2n}$ with a non-degenerate complex bilinear form $b$,
$Z$ is the space of $n$-dimensional $b$-isotropic subspaces, $\widehat{G}$ is
the $b$-orthogonal group $\mathrm{SO}_{2n}(\mathbb {C})$ and 
$G$ is the complex orthogonal group $\mathrm{SO}_{2n-1}(\mathbb {C})$ which is embedded
in $\hat{G}$ as the connected component at the identity of the isotropy group
of the $\hat{G}$-action at some nonzero point in $\mathbb {C}^{2n}$.
\end {enumerate}
\end {theo}
Referring to the above list of exceptions as Onishchik's list, our main result can be
stated as follows.
\begin {theo}\label{main theorem}
If $D$ is a $G_0$-flag domain in $Z=G/Q$, then $\mathrm{Aut}(D)$ can be described
as follows:
\begin {enumerate}
\item
If $\mathcal{O}(D)\not=\mathbb {C}$, or equivalently if it is holomorphically convex, 
$D$ is a product $\widetilde{D}\times F$ of a Hermitian symmetric space $\widetilde{D}$ 
of non-compact type and a compact flag manifold $F$, and 
$\mathrm{Aut}(D)$ is correspondingly a semidirect product
$\mathrm{Aut}(D)=\mathrm{Hol}(D, \mathrm{Aut}(F))\rtimes \mathrm{Aut}(\widetilde{D})$.
\item
If $\mathcal{O}(D)=\mathbb{C}$, then $\mathrm{Aut}(D)$ is a finite-dimensional
Lie group of holomorphic transformations on $D$ and, if the complexification
$G$ is the full group $\mathrm{Aut}(Z)^0$, then $\mathrm{Aut}(D)^0=G_0$.
\item
If $\mathcal{O}(D)=\mathbb {C}$ and $G$ is a proper subgroup of 
$\widehat{G}=\mathrm{Aut}(Z)^0$, then in each case of Onishchik's list 
$\mathrm{Aut}(D)^0=\widehat{G}_0$ is a uniquely determined real form
of $\widehat{G}$ which contains $G_0$ as a proper subgroup.
\end {enumerate}
\end {theo} 
It should be remarked that the simple proof given here of the fact that
if $\mathcal{O}(D)=\mathbb {C}$, then $\mathrm{Aut}(D)^0$ is a Lie group
acting on $Z$ does not yield a proof that in this case full group $\mathrm{Aut}(D)$
is a Lie group.  At the present time we have no other proof of this fact other
than that in \cite{H1}.  
\section{Cycle connectivity}\label{connectedness}
In \cite{H2} we used chains of cycles to study the pseudoconvexity and
pseudoconcavity of flag domains .  We continued the use of these
chains in our study of $\mathrm{Aut}(D)$ in \cite{H1}. Here, in particular
compared to the chains in \cite{K}, it is sufficient to consider chains of a 
very special type which we now introduce.

\bigskip\noindent
A basic fact, which is the tip of the iceberg of Matsuki duality, is that for
a flag domain $D$ any given maximal compact subgroup $K_0$ of $G_0$ has exactly
one orbit $C_0=K_0.z_0$ in $D$ which is a complex submanifold.  In fact it is
the (unique) orbit of minimal dimension.  If $K$ is the complexification of $K_0$, then
since $C_0$ is complex, $K$ stabilizes it.   Denoting $q:=\mathrm{dim}_{\mathbb {C}}C_0$,
we usually regard $C_0$ as a point in the Barlet cycle space $\mathcal{C}_q(D)$, but  for our
purposes here we may regard it as a point in the full Chow space $\mathcal{C}_q(Z)$
where $G$ is acting algebraically.   The group theoretical cycle space of $D$
is then defined as connected open subset
$$
\mathcal{M}(D)=\{g(C_0): g\in G, \ g(C_0)\subset D\}^0
$$
of the orbit of the base cycle $C_0$.
One can show that $\mathcal{M}(D)$ is a closed submanifold of $\mathcal{C}_q(D)$
(See \cite{FHW} for background and a systematic study of these cycle spaces.).   For
the purposes of this paper a chain of cycles is a finite connected union of (supports of) cycles
in $\mathcal {M}(D)$.  We often write such a chain as $(C_1,\ldots ,C_m)$ to indicate
that $C_i\cap C_{i+1}\not=\emptyset$  Using such chains we have the \emph{cycle connection} 
equivalence relation
$$
x\sim y \ \Leftrightarrow \ x \ \text{and} \ y \ \text{are contained in a chain}\,.
$$
Note that this relation is $G_0$-equivariant.  In particular, if $D=G_0/H_0$ then
there is a (possibly not closed) subgroup $I_0$ of $G_0$ which contains $H_0$
so that the quotient of $D$ by this equivalence relation is given by
$G_0/H_0\to G_0/I_0$.  Now if $z_0\in D$ is the base point where
$H_0:=G_{z_0}$ and $K_0.z_0=K.z_0=C_0$ is the base cycle, then, since
$C_0$ is by definition contained in the equivalence class of $z_0$, it is
immediate that $I_0\supset K_0$.  Since $K_0$ is a maximal subgroup
of $G_0$, i.e., any (not necessarily closed) subgroup of $G_0$ which contains
$K_0$ is either $K_0$ or $G_0$, the following is immediate (see also \cite{H1} 
and \cite{H2} for the same proof).
\begin {prop}\label{connected}
The following are equivalent:
\begin {enumerate}
\item
$\mathcal{O}(D)=\mathbb {C}$
\item
$D$ is not holomorphically convex.
\item
There is no nontrivial $G_0$-equivariant holomorphic map of 
$D$ to a Hermitian symmetric space $\widetilde{D}$ of noncompact type.
\item
$D$ is cycle connected.
\end {enumerate}
\end {prop}
\begin {proof}
The equivalence of the first three conditions follows from the discussion in
$\S\ref{intro}$.   If $D$ is cycle connected, then, since $\widetilde {D}$ is Stein
and therefore every holomorphic map to $\widetilde {D}$ is constant along every
chain, 4.) $\Rightarrow $ 3.).  Conversely, if $D$ is not cycle connected, then the
equivalence class containing the base point $z_0$ is just the cycle $C_0$ which is therefore
stabilized by the $G$-isotropy $P$ as well as $K$.  Since the cycle connection reduction is
given by $G_0/H_0\to G_0/K_0$, it follows that this fibration is the restriction of
the fibration $G/P\to G/\widetilde{P}$ of $Z$ where $\widetilde{P}=KP$ and therefore
the base $G_0/K_0$ is the Hermitian symmetric space $\widetilde{D}$.  In other words,
the cycle connected reduction is just the holomorphic reduction and in particular
$\mathcal{O}(D)\not=\mathbb{C}$.
\end {proof}
\noindent
\textbf{Remark.} For applications in another context, Griffiths, Robles and Toledo recently
gave another proof a result which is essentially equivalent to Proposition \ref{connected}.
(see \cite{GRT}).

\bigskip\noindent
Although it is well-known that $K_0$ is a maximal subgroup of $G_0$, for the convenience
of the reader we would like to give the following nice proof of J. Brun which was pointed
out to us by Keivan Mallahi Karai (see the Appendix of \cite{B}).
\begin {theo}
If  $G$ is a connected simple Lie group, $K$ is a  maximal compact
subgroup and $L$ is an abstract group which contains $K$, then
$L$ is either $G$ or $K$.
\end {theo}
\begin {proof}
Standard results in the theory of symmetric spaces show that $K$ is connected and
the adjoint representation of $K$ on $\mathfrak{g}/\mathfrak{k}$ is irreducible.
Thus if $\mathfrak{\ell}$ is a Lie subalgebra of $\mathfrak{g}$ which properly
contains $\mathfrak{k}$, then $\mathfrak{\ell}=\mathfrak{g}$.  Thus, if $L$ is
closed, then the result is immediate.   Furthermore, if $L$ is not closed and properly
contains $K$, then
its closure $c\ell(L)$ is the full group $G$.   In that case we let
$\mathfrak {k}'$ be the vector subspace of the Lie algebra $\mathfrak {g}$
of $G$ which is generated by  $\mathrm{Ad}(x)(\mathfrak {k})$ for all $x\in L$.
Since $c\ell (L)=G$, it follows by continuity that $\mathfrak{k}'$ is $G$-invariant
and since $G$ is simple, it is immediate that $\mathfrak{k}'=\mathfrak{g}$.
Therefore there are finitely many elements $x_i\in L$ so that
$$
\mathfrak{g}=\sum_1^m \mathrm{Ad}(x_i)(\mathfrak {k})
$$
and as a result the map 
$$
K^m\to G, \ K(k_1,\ldots, k_m)\mapsto \prod (x_ik_ix_i^{-1})
$$
has maximal rank at the origin.  Thus $L$ contains a compact neighborhood
of the origin and is therefore compact, i.e., contrary to assumption $L=K$.
\end {proof}
\section {Finiteness Theorem}
\noindent
Our original goal in this setting was to show that a flag domain $D$ is
either pseudoconvex or psedoconcave (\cite{H2}).  More precisely, we had
hoped to show that if $D$ is not holomorphically convex, then $C_0$ has
a pseudoconcave neighborhood which is filled out by cycles.  If this would
be possible, then using Andreotti's finiteness theorem (\cite{A}) we would
be able to conclude that the space of sections of any holomorphic vector
bundle, in particular the space $\mathrm{Vect}_{\mathcal {O}}(D)$, is finite-dimensional.
Although we have been successful in constructing such a neighborhood in
a number of cases (\cite{H2}), we have failed do this in general.   Recently, in
a substantially more general setting, Kollar proved the desired finiteness theorem
along with a number of equivalent properties which would follow from the
pseudoconcavity of $D$ (\cite{K}).  Here we make use of Kollar's result, leaving
the question of existence of the pseudoconcave neighborhood open.  

\bigskip\noindent
Formulated in our setting, Kollar's finiteness result can be stated as follows.
\begin {theo}
The space $\Gamma(D,E)$ of sections of any holomorphic vector bundle on
a cycle connected flag domain is finite-dimensional.
\end {theo}
\noindent
This is an immediate consequence of the same result for line bundles which
in turn is proved using the following Lemma (Lemma 15 in \cite{K}), again
formulated in our restricted context.
\begin {lem} 
Let $L$ be a holomorphic line bundle on $D$. Then, given $d\in \mathbb{N}$
there exists $d_0\in \mathbb{N}$ so that for every $C\in \mathcal{M}_D$ and
any $z_0\in C$ every section $s\in \Gamma(D,L)$ which vanishes of order $d_0$ at
$z_0$ vanishes of order $d$ along $C$.
\end {lem}
\noindent
The proof is given by classical methods which are reminiscent of Siegel's Schwarz Lemma.
One key point is that $C$ can be filled out by rational curves which in our case are
closures of orbits of $1$-parameter groups. 

\bigskip\noindent
Now, given a chain of cycles $(C_1,\ldots ,C_m)$ with $z_i\in C_i\cap C_{i+1}$,
and given $d_m\in \mathbb {N}$ we apply the Lemma to obtain $d_{m-1}\in \mathbb{N}$
so that if $s$ vanishes of order $d_{m_1}$ at $z_{m-1}$, then it vanishes of order
$d_m$ along $C_m$.  Working backwards to the first cycle in the chain, we
see that the Lemma holds for chains.
\begin {cor}\label{vanishing}
Given $d\in \mathbb{N}$ there exists $d_1\in \mathbb {N}$ so that
for any chain $(C_1,\ldots ,C_m)$ of length $m$ and any $z_1\in C_1$ if $s$ vanishes
of order $d_1$ at $z_1$, then it vanishes of order $d$ along $C_m$.
\end {cor}
\noindent
It should be emphasized that for a fixed $d$ the required vanishing order $d_1$
depends on $m$.  Thus to apply this result we need some sort of uniform estimate
for the length of a chain connecting two given points.  This can be given as follows.

\bigskip\noindent
For example, let $C_1$ be a base cycle for a given maximal compact subgroup
$K_0$.  Recall that the complexification $K$ has only finitely many orbits in $Z$
and therefore has a (unique) open dense orbit $\Omega $.  Take $z_1\in C_1$
and any point $z\in \Omega $ and let $(C_1,\ldots ,C_m)$ be a chain connecting
$z_0$ to $z$.  For $k\in K$ sufficiently close to the identity, the chain
$(k(C_1),\ldots ,k(C_m))$ is still contained in $D$.  Thus, since $k(C_1)=C_1$
and $k(z)$ can be an arbitrary point in a sufficiently small neighborhood
$U$ of $z$, we have the desired vanishing theorem.
\begin {cor}
If $s\in \Gamma(D,L)$ vanishes of sufficiently high order at a given point 
$z_1\in C_1$, then it vanishes identically.  In particular, $\Gamma(D,L)$
is finite-dimensional.
\end {cor}
\begin {proof}
Since the required vanishing order $d_1$ only depends on the number
$d_m$ and the length $m$, Corollary \ref{vanishing} implies that if
$s$ vanishes of order $d_1$ at $z_1$, then it vanishes at every point
of the set $U$ which was constructed above.  The desired result then
follows from the identity principle.
\end {proof}
\noindent
As we remarked above, the finiteness theorem for vector bundles is
an immediate consequence of this Corollary (see \cite{K}, p. 8).
\section{Integrability of vector fields}
The following is the main result of this section.
\begin {theo}
Let $Z=G/Q$ be a complex flag manifold and $\widehat{\mathfrak{g}}$ be a finite-dimensional
complex Lie algebra which contains $\mathfrak{g}:=\mathrm{Lie}(G)$.
Let $\widehat{G}$ be a complex Lie group which contains $G$ and is 
associated to $\widehat{\mathfrak{g}}$.
If $\widehat{\mathfrak {q}}$ is a complex subalgebra of $\widehat{\mathfrak {g}}$ so
that the quotient map 
$\widehat{\mathfrak{g}}\to \widehat{\mathfrak {g}}/\widehat{\mathfrak {q}}$ induces
an isomorphism
$$
\widehat{\mathfrak {g}}/\widehat{\mathfrak {q}}=\mathfrak{g}/\mathfrak {q}\,,
$$
then $\widehat{G}$ acts holomorphically  on $Z$ with
$$
Z=\widehat{G}/\widehat{Q}=G/Q\,.
$$ 
\end {theo}

\begin {proof}
We apply a basic idea of Tits.  For this regard $x_0:=\widehat{\mathfrak {q}}$ as a point
in the Grassmannian $X:=\mathrm{Gr}_k(\widehat{\mathfrak{g}})$ of subspaces of
dimension $k=\mathrm{dim}_{\mathbb {C}}\widehat{\mathfrak {q}}$ in $\widehat{\mathfrak{g}}$.
The isotropy group at $x_0$ of the $\widehat{G}$-action on $X$ is the normalizer
$$
\widehat{N}=\{\widehat{g}\in \hat{G}: \mathrm{Ad}(\hat{g})(\widehat{\mathfrak{q}})=\widehat{\mathfrak{q}}\}\,.
$$  
Denote by $N=\widehat{N}\cap G$ the $G$-isotropy at $x_0$ and note that if $g\in N$
and $\xi\in \mathfrak{q}$, it follows that 
$\mathrm{Ad}(g)(\xi)\in \widehat{\mathfrak {q}}\cap \mathfrak{g}=\mathfrak{q}$.  In other words
$N$ is contained in the normalizer of $\mathfrak{q}$ in $\mathfrak {g}$.  Since the parabolic
group $Q$ is self-normalizing in $G$, it follows that $N\subset Q$.  But 
$\widehat{\mathfrak{n}}\supset \widehat{\mathfrak{q}}$ and $\widehat{\mathfrak{q}}\cap \mathfrak{g}=
\mathfrak{q}$.  Therefore $\mathfrak{n}\supset \mathfrak{q}$.  Consequently $N=Q$
and the $G$-orbit of $\widehat{\mathfrak{q}}$ is the compact manifold $Z=G/Q$.  Since 
$\widehat{\mathfrak {g}}/\widehat{\mathfrak{q}}=\mathfrak{g}/\mathfrak{q}$, the $\widehat{G}$-orbit
of $\widehat{\mathfrak{q}}$ has dimension at most that of $G/Q$.  But on the other hand
$\widehat{G}\supset G$ and therefore the $\widehat{G}$-orbit has the same dimension
as the $G$-orbit.  Consequently $G.x_0$ is open in $\widehat{G}.x_0$ and the
compactness of $G.x_0$ implies that these  orbits agree.
\end {proof}
\noindent
Applying the Finiteness Theorem, the following is now immediate.
\begin {cor}
Let $G_0$ be a simple real form of a complex semisimple Lie group $G$
and $D$ be a cycle connected $G_0$-flag domain in a $G$-flag manifold $Z=G/Q$.  
Let $\widehat{\mathfrak{g}}$ be the Lie algebra  of holomorphic vector fields on $D$. 
Then the restriction mapping $R:\mathfrak{aut}(Z)\to \widehat{\mathfrak{g}}$ is an
isomorphism and the action of $\widehat{\mathfrak {g}}$ can be integrated to
the action of a connected complex Lie group $\widehat{G}$ which is thereby identified
with $\mathrm{Aut}(Z)^0$.
\end {cor}
\noindent
As a consequence we have the description of $\mathrm{Aut}(D)$ which
was proved by other means in \cite{H1}.
\begin {cor}
If $D$ is a $G_0$-flag domain $Z=G/Q$, then one of the following holds:
\begin {enumerate}
\item
If $D$ is holomorphically convex, it is a product of a compact flag
manifold and Hermitian symmetric space of noncompact type and
$\mathrm{Aut}(D)^0$ can be described as in $\S\ref{intro}$.
\item
If $D$ is not holomorphically convex or equivalently it is cycle connected,
then $\mathrm{Aut}(D)^0$ is a finite-dimensional Lie group which is
acting on $Z$ and agrees with $G_0$ with the possible exceptions
in the situations classified by Onishchik where $G$ is a proper subgroup
of $\widehat{G}=\mathrm{Aut}(Z)^0$.  
\end {enumerate}
\end {cor}
\section{Exceptional cases}
\noindent
To complete our project of understanding the automorphism groups of flag domains,
we must analyze the exceptional cases indicated in the above Corollary.  We do
this here, proving the following result.
\begin {theo}\label{exceptional}
Suppose that $D\subset Z=G/Q$ is a $G_0$-flag domain which is
not holomorphically convex and that $G$ is properly contained in
complex Lie group $\widehat{G}=\mathrm{Aut}(Z)^0$. 
Then there is a uniquely determined real form 
$\widehat{G}_0=\mathrm{Aut}(D)^0$ of 
$\widehat{G}$ which properly contains $G_0$ and which stabilizes $D$. 
\end {theo}
\noindent
Our proof of this fact amounts to a concrete discussion for each of
the three classes of exceptions in Onishchik's list which was
given in $\S\ref{intro}$. Below we show that these cases not only occur 
but also occur at the level of real forms.  This is the content of (3) in 
Theorem \ref{main theorem} and Theorem \ref{exceptional} above.  
 
\subsection{Projective space}\label{projective space}
\noindent
Here we consider the case where $Z=\mathbb {P}(V)$ is the projective
space of an even-dimensional complex vector space $V=\mathbb {C}^{2n}$.
Define the complex bilinear form $b$ by $b(z,w)=z^tw$.  In the standard basis
$(e_1,\ldots ,e_{2n})$ define $J:V\to V$ by
$J(e_i)=e_{n+1}$, $i\le n$, and $J(e_i)=-e_{i-n}$, $i>n$.  Note $J$ is $b$-orthogonal
with $J^2=-\mathrm{Id}$ and define a
(complex, bilinear)  symplectic form by $\omega (z,w)=z^tJw$.  
Define $V_+:=\mathrm{Span}\{e_1,\ldots, e_n\}$ and $V_-=\mathrm{Span}\{e_{n+1},\ldots ,e_{2n}\}$
and correspondingly $E:=+\mathrm{Id}\oplus -\mathrm{Id}$. If
$C:V\to V$ denotes the standard complex conjugation given by
$z\mapsto \bar{z}$,  define a non-degenerate (mixed-signature) Hermitian structure 
on V by $h(z,w)=z^tEC(w)$.  Finally, if the antilinear map $\varphi :V\to V$ is 
defined by $z\mapsto -JECw$,  it follows that
$h(z,w)=\omega (z,\varphi(w))$ and, since $\varphi^2=-\mathrm{Id}$, 
that $\varphi $ is an $h$-isometry. Observe that if $P$ is 
a $\varphi $-invariant subspace of $V$, then $P^{\perp_h}=P^{\perp_\omega}$.
In particular, $P$ is symplectic if and only if it is $h$-nondegenerate and in either
of these cases $V=P\oplus P^\perp_h$ is a decomposition of $V$ into $h$-nondegenerate,
symplectic subspaces.

\bigskip\noindent
The complex symplectic group $G=\mathrm{Sp}_{2n}(\mathbb {C})$ defined by
$\omega $ has two types
of real forms.  The first case to be considered is where 
$\widehat{G}_0$
is the real form $\mathrm{SU}(n,n)$ of $\widehat{G}=\mathrm{SL}_{2n}(\mathbb {C})$
which is defined as the group of $h$-isometries.  In this case the real form 
$G_0=\mathrm{Sp}_{2n}(\mathbb {R})$ of $G=\mathrm{Sp}_{2n}(\mathbb {C})$ 
is defined as the intersection
$\widehat {G}_0\cap \mathrm {Sp}_{2n}(\mathbb {C})$.
Considering the orbits of these groups on $\mathbb {P}(V)$ we let
$D_+$ (resp. $D_-$) be the open sets in $\mathbb {P}(V)$ of $h$-positive (reps. $h$-negative)
lines.
\begin {prop}\label{Sp_2n(R)}
The open sets $D_+$ and $D_-$ are both orbits of $G_0$ and $\widehat {G}_0$.
\end {prop}
\begin {proof}
It is clear that $D_+$ and $D_-$ are $G_0$- and $\widehat {G}_0$-invariant and
that $D_+\cup D_-$ is dense in $Z$.
Since $G_0\subset \widehat {G}_0$, it is therefore enough to show that $G_0$
acts transitively on both sets.  The proof for $D_+$ is exactly the same as
for $D_-$ and therefore we only give it for $D_+$.  For this, given positive lines
$L=\mathbb {C}.z$ and $\widetilde{L}=\mathbb {C}.\widetilde {z}$, we define 
$P=\mathrm{Span} \{z,\varphi(z)\}$ and 
$\widetilde{P}=\mathrm{Span}\{\widetilde{z},\varphi(\widetilde{z})\}$.
These planes are $h$-nondegenerate and symplectic.  We normalize $z$ and $\widetilde{z}$
so that $\Vert z\Vert^2_h=\Vert \widetilde{z}\Vert^2=1$ and, since $\varphi :E_+\to E_-$,
$\Vert \varphi(z)\Vert^2=\Vert \varphi(\widetilde{z})\Vert^2=-1$.  Applying this procedure
to $P^\perp$ and $\widetilde {P}^\perp$ we have $h$- and $\omega $-orthogonal decompositions
$$
V=P_1\oplus \ldots \oplus P_n=\widetilde{P}_1\oplus \ldots \oplus \widetilde{P}_n
$$
of $V$.   Futhermore, every $P_i$ (resp. $\widetilde{P}_i$) comes equipped with
a basis $(z_i,\varphi(z_i))$ (resp. $(\widetilde{z}_i,\varphi(\widetilde{z}_i))$ such that
the mapping $T_i:P_i\to \widetilde{P}_i$ defined by $z_i\to \widetilde{z}_i$ and 
$\varphi(z_i)\mapsto \varphi(\widetilde {z}_i)$ is both symplectic and an 
$h$-isometry.  It follows that $T=T_1\oplus \ldots \oplus T_n$ is both a symplectic
isomorphism and $h$-isometry of $V$, i.e., $T\in G_0$.  Since $T(L)=\widetilde{L}$,
the proof is complete.
\end {proof}  
\noindent
Now let us turn to the real form $G_0=\mathrm{Sp}(2p,2q)$ of 
$G=\mathrm{Sp}_{2n}(\mathbb {C})$.  In this case we line up $J$ and
$E$ in a different way. The decomposition $V:=V_+\oplus V_-$ and $J$ are
the same, but now $h$ has signature $(p,q)$ on both spaces, being defined 
by the block diagonal matrix $E_{p,q}=(\mathrm{Id}_p,-\mathrm{Id}_q)$.
Then $\widehat{G}_0=\mathrm{SU}(2p,2q)$ is
defined as above by the Hermitian form $h$ and
$G_0=G\cap \widehat{G}_0$.  The proof of the following
fact is exactly the same as that of Proposition \ref{Sp_2n(R)} above.

\bigskip\noindent
\textbf{Zusatz.} Proposition \ref{Sp_2n(R)} also holds for $G_0=\mathrm{Sp}(2p,2q)$
and $\widehat{G}_0=\mathrm {SU}(2p,2q)$.\qed
\subsection{$5$-dimensional quadric}
Here we consider $V=\mathbb {C}^7$ equipped with the complex
bilinear form $b$ defined by $\Vert z\Vert^2_b=(z_1^2+z_2^2+z_3^2)-(z_4^2+\ldots +z_7^2)$
and Hermitian form $h$ defined by  
$\Vert z\Vert^2_h=(\vert z_1\vert^2+\vert z_2\vert^2+\vert z_3^2\vert^2)
-(\vert z_3\vert^2+\ldots +\vert z_7\vert^2)$\,. 
Denote by $\widehat {G}=\mathrm{SO_7}(\mathbb {C})$ the associated
complex orthogonal group and by $\widehat{G}_0:=\mathrm{SO}(3,4)$ the
associated group of Hermitian isometries.   

\bigskip\noindent
We regard the exceptional complex Lie group $G=G_2$ as being
embedded in $\widehat{G}$ as the automorphism group $\mathrm{Aut}(\mathbb {O})$
of the octonians.  It has a unique noncompact real form 
$G_0=\mathrm{Aut}(\widetilde {\mathbb {O}})$, the automorphism group of the split
octonians $\widetilde {\mathbb{O}}$.  In this way $G_0$ is the intersection 
$G\cap \widehat{G}_0$ of $G$ with the  real from $\widehat{G}_0=\mathrm{SO}(3,4)$ 
(see,e.g., [Ha] for details).   Note that $\widehat{G}_0$ is 
invariant by the standard complex conjugation $z\mapsto \bar{z}$.

\bigskip\noindent
The remainder of this paragraph is devoted to the proof of
the following fact.
\begin {prop}\label{equality}
For every $z\in Z$ it follows that $G_0.z=\widehat{G}_0.z$.  In particular the 
open orbits of $G_0$ and $\widehat{G}_0$ coincide.
\end {prop}
\noindent
We should note that, as indicated below, the open orbits of $\widehat{G}_0$ are the
spaces $D_+$ and $D_-$ of positive and negative lines, respectively.

\bigskip\noindent
For the proof of Proposition \ref{equality} 
we use Matsuki duality (see, e.g., Chapter 8 in Part II of \cite{FHW})
which states that there is a $1-1$ correspondence between the $G_0$-orbits 
and $K$-orbits in $Z$.  This can be given as follows:  For every $G_0$-orbit 
there is a unique $K$-orbit which intersects it in the unique $K_0$-orbit of minimal
dimension and vice versa, i.e., given a $K$-orbit there is a unique $G_0$-orbit
which intersects it in the unique $K_0$-orbit of minimal dimension.  Due to our
interest in the open $G_0$-orbits (resp. $\widehat{G}_0$-orbits) in $Z$, we have 
stated the above result on that side of the duality.  However, we have found
it more convenient to prove the corresponding dual statement.

\bigskip\noindent
Let us fix the maximal compact subgroup 
$K_0\cong (\mathrm{SU}_2\times \mathrm{SU}_2)/(-\mathrm{Id},-\mathrm{Id})$ of $G_0$ 
being diagonally embedded in the maximal
compact subgroup $\widehat{K}_0=S(\mathrm{O}(3)\times \mathrm{O}(4))$ of $\widehat{G}_0$. 
If $E_+:=\mathrm{Span}\{e_1,e_2,e_3\}$ and $E_-:=\mathrm{Span}\{e_4,\ldots ,e_7\}$,
Then we define $z_+:=e_1+ie_2$, $z_-:=e_4+ie_5$ and observe that the
base cycles $C_+$ and $C_-$ for the open orbits of the $\widehat{G}_0$-action are
the quadrics of $b$-isotropic lines in $E_+$ and $E_-$, respectively.  The corresponding
open orbits are the spaces $D_+=\widehat{G}_0.z_+$ of positive lines in $Z$ and
$D_-=\widehat{G}_0.z_-$ of negative lines, respectively.  The complement of
$D_+\cup D_-$, which is the space of lines which are both $b$- and $h$-isotropic,
consists of two $\widehat{G}_0$-orbits, the real points $Z_{\mathbb {R}}$ and its complement.

\bigskip\noindent
The $\widehat{K}$-orbits which correspond via Matsuki duality to the four $\widehat{G}_0$-orbits
are the two base cycles $C_+$ and $C_-$, the open $\hat{K}$-orbit of any point
on $Z_{\mathbb R}$ and a forth orbit $\mathcal {O}$ which has two ends, i.e., which has the
two base cycles on its boundary.  In fact this forth orbit is a $\mathbb {C}^*$-principal
bundle over the $2$-dimensional cycle $C_-$ (See \cite{FHW}, $\S16.4$ for a detailed
discussion in the case of the $K3$-period domain which can be transferred verbatim to
the case at hand.).  To prove the above Proposition \ref{equality}
we show that $K$ acts transitively on each of these four $\widehat{K}$-orbits.

\bigskip\noindent
Now the second factor of $\widehat{K}$ acts trivially on $C_+$ and the first factor
acts trivially on $C_-$ and vice versa.   Since $K$ is diagonally embedded in $\widehat{K}$ and
projects onto both factors, it is immediate that it acts transitively on both
$C_+$ and $C_-$ as well.  Since $\mathcal {O}$ is a $\mathbb {C}^*$-bundle
over $C_-$, $K$ acts transitively on the base of this bundle and has an
open orbit in the bundle space $\mathcal {O}$, it is immediate that it acts transitively
on $\mathcal {O}$.  

\bigskip\noindent
It remains to show that $K$ acts transitively on the open $\hat{K}$-orbit.  
For this we first note that, since $e_3+e_4\in Z_{\mathbb {R}}$ and the connected
component at the identity of $\hat{K}_0$ is the product of the special orthogonal
groups of $E_+$ and $E_-$, it follows that up to finite group quotients 
$Z_{\mathbb {R}}$ is the corresponding product $S^2\times S^3$ of spheres.
One immediately observes that $G_0$ acts transitively on $Z_{\mathbb {R}}$, because
every $G_0$-orbit is at least half-dimensional over $\mathbb {R}$. Thus $K_0$ acts
transitively on $Z_{\mathbb {R}}$ and if $z_{\mathbb {R}}$ is an arbitrary point
of $Z_{\mathbb {R}}$, it follows that $K.z_{\mathbb {R}}$ is open in $\hat{K}.z_{\mathbb {R}}$.

\bigskip\noindent
To complete the proof of Proposition \ref{equality} we must show that 
$K.z_{\mathbb {R}}=\widehat{K}.z_{\mathbb {R}}$.  For this we let $\widehat{K}_1$ be the first
factor of the product decomposition of the connected component of $\widehat{K}$
and consider the homogeneous fibration
$$
\widehat{K}.z_{\mathbb {R}}=\widehat{K}/\widehat{L}\to \widehat{K}/\widehat{K}_1\widehat{L}
=\widehat{K}_2/\widehat{L}_2=B\,,
$$
Since $Z_{\mathbb {R}}$ is essentially a product $S^2\times S^3$ corresponding
to the decomposition of the connected component $\widehat{K}_0$, it follows that up to 
finite group quotients the base $B$ is the complexification of $S^3$,
i.e., the affine quadric $Q_{(3)}=\mathrm {SO}_4(\mathbb {C})/\mathrm{SO}_3(\mathbb {C})$.
Since $K$ projects surjectively onto both factors of $\widehat{K}$, it is immediate that
$K$ acts transitively on $B=K/M$.  Now the induced fibration of $K.z_{\mathbb {R}}$ 
is a homogeneous bundle $K/L\to K/M$ where the fiber $M/L$ is an open 
$M$-orbit in the corresponding fiber $\widehat{F}$ of the $\widehat{K}$-bundle 
$\widehat{K}/\widehat{L}\to \widehat{K}/\widehat{M}$.  But $\widehat{M}$ acts on this fiber as 
$\mathrm{SO}_3(\mathbb {C})$ so that $\hat{F}$ is the affine quadric $Q_{(2)}$.
Since $K/M$ is affine, $M$ is reductive.  But the only reductive subgroup
of $\mathrm{SO}_3(\mathbb {C})$ with an open orbit in $Q_{(2)}$ is $\mathrm{SO}_3(\mathbb {C})$
itself.  Consequently $K$ does indeed act transitively on the open $\widehat{K}$-orbit and
the proof of Proposition \ref{equality} is complete.
 \qed

\subsection{Space of isotropic $n$-planes in $\mathbb {C}^{2n}$}
Now let $\widehat{V}=\mathbb {C}^{2n}$ be equipped with its standard basis
$(e_1,\ldots, e_{2n})$ and complex bilinear form defined by $b(z,w)=z^tw$.
The complex orthogonal group $\mathrm {SO}_{2n}(\mathbb {C})$ of $b$-isometries 
is denoted by $\widehat{G}$.  We let 
$G:=\mathrm{Fix}_{\widehat{G}}(e_{2n})$.  In this way 
$G\cong \mathrm{SO}_{2n-1}(\mathbb {C})$ is the orthogonal group of the
the restriction of $b$ to $V:=\mathrm{Span}\{e_1,\ldots ,e_{2n-1}\}$.  
We consider the action of these groups 
on the flag manifold $Z$ of $n$-dimensional $b$-isotropic subspaces of $\hat{V}$.  
\begin {prop}
The groups $G$  and $\widehat{G}$ act transitively on $Z$.
\end {prop}
\begin {proof}
Note that the intersection $W:=\widehat{W}\cap V$ of an isotropic $n$-plane in $\widehat{V}$
is an isotropic $(n-1)$-plane in $V$. It follows that $\widehat{W}=W\oplus \mathbb {C}.(v+ie_{2n})$
for some $v\in V$.  Applying an appropriate element of $G$, we may assume that
$v=e_{2n-1}$ and it then follows that $W\subset\mathrm{Span}\{e_1,\ldots, e_{2n-2}\}$.  We then
apply the induction assumption to obtain a transformation in the corresponding
$\mathrm {SO}_{2n-2}(\mathbb {C})$ to bring $W$ to the normal form with basis
$(e_1+ie_{n+1},\ldots ,e_{n-1}+ie_{2n-2})$ so that altogether we have found a transformation
in $G$ which brings $W$ to the normal form with the basis
$(e_1+ie_{n+1},\ldots ,e_{2n-1}+ie_{2n})$.
\end {proof}
\noindent
Recall that up to conjugation the only real forms of $\mathrm{SO}_{2n-1}(\mathbb {C})$ are the
isometry groups $G_0=\mathrm{SO}(p,q)$ for the mixed signature Hermitian form
defined by $h(z,w)=z^tEC(w)$ on $V$ where $E=E_{p,q}$ is defined in the same
way as in $\S\ref{projective space}$. Without loss of generality we may choose $h$ to
be this form and note that an appropriately chosen arbitrarily small
perturbation of an isotropic $n$-plane $\widehat {W}$ will result in the intersection
$W=\widehat{W}\cap V$ being $h$-nondegenerate.  Thus, if
$G_0.z=:D$ is an open orbit in $Z$, the $(n-1)$-plane $W$ associated to $z$ 
is $h$-nondegenerate. 

\bigskip\noindent
Note that if $p$ is even, then $q$ is odd and vice versa.  To make the notation more
explicit, we assume that $p$ is even. Now the space of $h$-positive
$b$-isotropic lines in $V$ is an open $G_0$-orbit. Thus, given $W$ as above, we may 
apply an element $g\in G_0$ so that after replacing $W$ by $g(W)$ we have 
$L=\mathbb{C}(e_1+ie_2)\subset W$.  Notice that subspace of $V$ of vectors which are
both  $h$- and $b$-orthogonal to $L$ is simply $\mathrm{Span}\{e_3,\ldots ,e_{2n-1}\}$.  Thus, after going
to this smaller space, we have the same situation as before.  Hence we may continue on
by induction to obtain a maximal $h$-positive subspace $W_+$ of $W$ which is 
$\frac{p}{2}$-dimensional and which has a distinguished basis produced by our procedure.  
Applying the same argument as above
to the $h$-complement $W_+^\perp$ in $W$, one obtains an element $g\in G_0$ so that
$W_0:=g(W)$ has the distinguished basis 
$$
(e_1+ie_2,e_3+ie_4,\ldots ,e_{p-1}+ie_p,e_{p+1}+ie_{p+2},\ldots e_{2n-3}+ie_{2n-2})\,.
$$
\begin {prop}\label{normalform}
If $\widehat{W}$ is a $b$-isotropic $n$-plane in $\widehat{V}$, then there exists
an element $g\in G_0$ with $g(\widehat{W})=W_0\oplus \mathbb {C}(e_{2n-1}+ie_{2n})=:\widehat{W}_0$.
\end {prop}
\begin {proof}
Let $g\in G_0$ be chosen as above with $g(W)=W_0$. It is then immediate that
$g(\widehat {W})=W_0\oplus \mathbb {C}\widehat{w}$ where $\widehat{w}=\pm e_{2n-1}+ie_{2n}$.
We obtain the positive sign by, e.g., multiplying $e_1$ and $e_2$ by $i$, $e_{2n-1}$ by
$-1$ and $e_j$ by $+1$ otherwise.  Since this transformation is also in $G_0$, the desired
result follows.
\end {proof}
\begin {theo} The Hermitian form $h$ can be naturally extended to a non-degnerate Hermitian
form $\widehat{h}$ on $\widehat{V}$ with signature $(p,q+1)$ (resp. $(p+1,q)$) if $p$ is even 
(resp.odd) so that the unique open orbit $D$ of the resulting real form $\widehat{G}_0$
is the set of isotropic $n$-planes of signature $(\frac{p}{2},\frac{q+1}{2})$ 
(resp.$(\frac{p+1}{2},\frac{q}{2})$). Furthermore, the $h$-isometry group $G_0$ in $\mathrm{SO}_{2n-1}(\mathbb {C})$
also acts transitively on $D$ which is also its unique open orbit in $Z$.
\end {theo}
\begin {proof} 
It is enough to consider the case we where $p$ is even and $\Vert e_{2n-1}\Vert^2_h=-1$. Extending
$h$ to $\widehat{h}$ on $\widehat{V}$ with $e_{2n}$ being orthogonal to $V$ and 
$\Vert e_{2n}\Vert^2=-1$, it follows that $\widehat{h}$ is of signature $(p,q+1)$.  Let
$\widehat{G}_0=\mathrm{SO}(p,q+1)$ be the real form of $\widehat{G}=\mathrm{SO}_{2n}(\mathbb {C})$
defined by $\widehat{h}$. Arguing as above we see that the unique open $\widehat{G}_0$-orbit $D$
in $Z$ is the set of isotropic $n$-planes $\widehat{W}$ with signature 
$(\frac{p}{2},\frac{q+1}{2})$. A reformulation of Proposition \ref{normalform} is that
$G_0$ also acts transitively on $D$.
\end {proof}
The following is a less technical formulation of this fact.
\begin {cor}
If $Z$ is the complex flag manifold of isotropic $n$-planes in 
$\mathbb{C}^{2n}$ where both
$\widehat{G}$ and $G$ act transitively, every real form $G_0$ of $G$ has
a unique open orbit $D$ which is the unique open orbit of a canonically
determined real form $\widehat{G}_0=\mathrm{Aut}(D)^0$ of $\widehat{G}=\mathrm{Aut}(Z)^0$.
\end {cor}

\section*{Acknowledgements} The research for this project was 
supported by SFB/TR 12 and SPP 1388 of the Deutsche Forschungsgemeinschaft.

\end {document}